\documentclass[12pt]{amsart}
\usepackage[latin1]{inputenc}
\usepackage{subfigure}
\usepackage{amssymb}
\usepackage{amsmath}
\usepackage{multirow}
\usepackage[Large]{caption}
\usepackage[table]{xcolor}
\newtheorem{theorem}{Theorem}
\newtheorem{propo}[theorem]{Proposition}

\newtheorem{cor}[theorem]{Corollary}
\newtheorem{eg}{Example}

\newcommand{\Classes}{\mathrm{Classes}}
\newcommand{\Num}{\#\Classes}
\newcommand{\Eq}{\mathrm{Eq}}
\newcommand{\fourdots}{{\begin{smallmatrix} 
     \cdot\kern -1.8pt & \cdot \\[-2.7pt]
     \cdot\kern -1.8pt& \cdot \end{smallmatrix}}}
\newcommand{\twolines}{{\bf \kern 0.6pt\vrule height6pt depth0pt width0.5pt
\kern 4pt \vrule height6pt depth0pt width0.5pt}}
\newcommand{\oursquare}{\square}
\newcommand{\ourstar}{*}
\newcommand{\inv}{\mathrm{inv}}
\newcommand{\Inv}{\mathrm{Inv}}
\newcommand{\floor}[1]{\lfloor #1 \rfloor}
\newcommand{\ceiling}[1]{\lceil #1 \rceil}
\newcommand{\cred}[1]{{\color{red}#1}}
\def\urltilde{\kern -.15em\lower .7ex\hbox{\~{}}\kern .04em}
\catcode`@=11
\def\mycases#1{\left\{\,\vcenter{\normalbaselines\m@th
    \ialign{$##\hfil$&\quad##\hfil\crcr#1\crcr}}\right.}
\catcode`@=12 

\date{7 November 2011}

\begin{document}
\author{Steven Linton}
\address[Steven Linton]{Centre for Interdisciplinary Research in Computational Algebra, \\
  University of St Andrews\\  
  St Andrews, Fife KY16 9AJ, Scotland\\
  }
\author{James Propp}
\address[James Propp]{University of Massachusetts\\  
  Lowell, MA 01854, USA\\}
\author{Tom Roby}
\address[Tom Roby]{University of Connecticut\\
  Storrs, CT 06269, USA\\}
\author{Julian West}
\address[Julian West]{University of Bristol\\  
Bristol BS8 1TW, England\\}
\thanks{First author supported by EPSRC grant EP/C523229/01.}
\thanks{Fouth author supported by NSERC operating grant OGP0105492.}
\title[Constrained Equivalence classes of permutations]{Equivalence Classes of
Permutations under Various Relations Generated by Constrained Transpositions}


\begin{abstract}
We consider a large family of equivalence
relations on permutations in $S_{n}$ that generalise those discovered by
Knuth in his study of the Robinson-Schensted correspondence.  In our
most general setting, two permutations are equivalent if one can be
obtained from the other by a sequence of pattern-replacing moves of
prescribed form; however, we limit our focus to patterns where two
elements are transposed, subject to the constraint that a third
element of a suitable type be in a suitable position.
For various instances of the problem, we compute the number of
equivalence classes, determine how many $n$-permutations are equivalent
to the identity permutation, or characterise this equivalence
class. Although our results feature familiar integer sequences (e.g.,
Catalan, Fibonacci, and Tribonacci numbers) and special classes
of permutations (layered, connected, and $123$-avoiding), some of the
sequences that arise appear to be new.

\end{abstract}

\maketitle

\section{Introduction and motivation}\label{sec:int}

We consider a family of equivalence relations on permutations in
$S_{n}$ in which two permutations are considered to be equivalent if
one can be converted into the other by replacing a short subsequence of
(not necessarily adjacent) elements by the same elements permuted 
in a specific fashion, or (extending by transitivity) by a sequence 
of such moves.  These generalise the relations discovered by Knuth 
in his study of the Robinson-Schensted correspondence, though the 
original motivations for this project were unrelated.  We begin the 
systematic study of such equivalence relations, connecting them with 
integer sequences both familiar and (apparently) new.  

Consider the following three examples of turning one 7-permutation into another
in which selected 3-subsequences (marked in \textbf{bold}) are re-ordered:
\begin{eqnarray}\label{eqn:eg1}
12\mathbf{34}56\mathbf{7} &\rightarrow & 12\mathbf{74}56\mathbf{3}\\ 
\mathbf{127}4563 &\rightarrow & \mathbf{721}4563\\ 
721\mathbf{456}3  &\rightarrow &721\mathbf{654}3
\end{eqnarray}
In each of these examples, a subsequence of \textit{pattern} 123 (i.e., 
a triple of not necessarily adjacent entries whose elements are in the 
same relative order as 123) is replaced by the same set of elements arranged 
in the pattern 321. Allowing replacements of a designated kind to be 
performed ad libitum, in reverse as well as forward, induces an equivalence
relation on the symmetric group $S_{7}$.  Accordingly we can say that the
permutations 1234567, 1274563, 7214563, and 7216543 are all equivalent
under the replacement $123 \leftrightarrow 321$.  

Interesting enumerative questions arise when the elements being replaced
are allowed to be in general position (Section~\ref{sec:gen}), when the
replacements are constrained to involve only adjacent elements
(Section~\ref{sec:adj}), and when replacements are constrained to affect
only subsequences of consecutive elements representing a run of
consecutive values (Section~\ref{sec:dbl}). Each of these respective
\emph{types} of replacements is illustrated in one of the three examples
above.  It will be convenient to group subsequences that are allowed to
replace one another into sets, e.g., describing the three permutations
above as being ``$\{123,321 \}$-equivalent''.  We may also wish to allow
more than one type of (bi-directional) replacement, such as both $123
\leftrightarrow 321$ and $123 \leftrightarrow 132$. If the intersection
of these sets is nonempty, the new relation can be described simply by
the union of the two sets: $\{123,132,321\} = \{123,321\} \cup
\{123,132\}$. To formalise this more generally we consider collections
of disjoint replacement sets that form a set partition of $S_{3}$; any
two patterns within the same set may replace one another within the
larger permutation to give an equivalent permutation.

Let $\pi \in S_{n}$, and let $P=\{B_{1},B_{2},\dots,
B_{t}\}$ be a (set) partition of $S_{k}$, where $k\leq n$.  Each block
$B_{l}$ of $P$ represents a list of $k$-length patterns which can
replace one another. We call two permutations $P^{\fourdots}$-equivalent if one
can be obtained from the other by a sequence of replacements, each
replacing a subsequence of pattern $\sigma_{i}$
with the same elements in the pattern $\sigma_{j}$, 
where $\sigma_{i}$ and $\sigma_{j}$ lie in the same block
$B_{l}$ of $P$.  Let $\Eq^{\fourdots}(\pi,P)$ denote the set of permutations
equivalent to $\pi$ under $P^{\fourdots}$-equivalence; e.g., 
1234567, 7214563, and $7216543\in
\Eq^\fourdots\left(1274563,\big\{\{123,321\}\big\}\right)$. 
Similarly we denote by $P^{\twolines}$ the equivalence
relation, and by $\Eq^{\twolines}(\pi,P)$ the equivalence
class of $\pi$, under replacement within $P$ only of adjacent elements; e.g., 
7214563 and $7216543 \in \Eq^\twolines\left(1274563,\big\{\{123,321\}\big\}\right)$. 
We use $P^{\oursquare}$ and $\Eq^\oursquare(\pi,P)$ when both positions and values are
constrained, e.g., 
$7214563 \in \Eq^\oursquare\left(7216543,\big\{\{123,321\}\big\}\right)$.  
To refer to such classes generically we use the notation
$\Eq^\ourstar(\pi,P)$.  The automorphism $\pi\mapsto \pi^{-1} $ replaces
adjacency of positions with adjacency of values; hence, for the enumerative
questions we treat, there is no need to separately consider a fourth
case where we only constrain values to be adjacent.
The set of distinct equivalence classes into which $S_n$ splits under
an equivalence $P^{\ourstar}$ is denoted by $\Classes^\ourstar(n,P)$.

The present paper begins the study of these equivalence relations by
considering three types of questions: 

(A) Compute the number of equivalence classes,
$\#\Classes^{\ourstar}(n,P)$, into which $S_{n}$ is partitioned.  

(B) Compute the size, $\#\Eq^\ourstar(\iota_n,P)$, of the equivalence
class containing the identity $\iota _{n}=123\cdots n$. 

(C) Characterise the set $\Eq^\ourstar(\iota_n,P)$ of
permutations equivalent to the identity.

Although the framework above allows for much greater generality, in this
paper we will mainly restrict our attention to replacements by patterns of
length $k=3$, and usually to replacement patterns built up from pairs in
which one permutation is the identity, and the other is a transposition
(i.e., fixes one of the elements). Omitting some cases by symmetry, 
we have the following
possible partitions of $S_{3}$, where (as usual) we omit singleton blocks:

\begin{eqnarray*}
P_1 = \big\{ \{123, 132\} \big\}, \\
P_2 = \big\{ \{123, 213\} \big\}, \\
P_4 = \big\{ \{123, 321\} \big\}.
\end{eqnarray*}

We will also consider applying two of these replacement operations 
simultaneously, and we will number the appropriate partitions as
\begin{eqnarray*}
P_3 = \big\{ \{123, 132, 213\} \big\}, \\
P_5 = \big\{ \{123, 132, 321\} \big\}, \\
P_6 = \big\{ \{123, 213, 321\} \big\},
\end{eqnarray*}
following the convention $P_{i+j} := P_i \vee P_j$, the \emph{join} of
these two partitions~\cite[ch.3]{StanleyEC1}. Indeed we
can allow all three replacements: $P_7 = \big\{ \{123, 132, 213, 321\} \big\}$.
(In fact, the cases $P_1$ and $P_2$ are equivalent by symmetry, as
are $P_5$ and $P_6$. We list $P_1$ and $P_2$ separately only
in order to consider their join.)

Our motivation for focussing attention on pairs of this form is that 
we can then think of an operation, not in terms of replacing one 
pattern by another, but simply in terms of \emph{swapping} two
elements within the pattern, with the third serving as a catalyst
enabling the swap.  In a followup~\cite{PRW11} to the current paper, the
authors treat the remaining (non-swapping) cases for all partitions of
$S_{3}$ consisting of exactly one non-singleton block which contains the
identity $123$. 

By far the best-known example of constrained swapping in permutations is
the \emph{Knuth Relations}~\cite{Knuth}, which allow the swap of
adjacent entries provided an intermediate value lies immediately to the
right or left.  In the notation of this paper, they correspond to
$P^{\twolines}_{K}=\big\{ \{ 213,231\}, \{132,312\}\big\}$. Permutations
equivalent under this relation map to the same first coordinate
($P$-tableau) under the Robinson-Schensted correspondence. 

Mark Haiman introduced the notion of \emph{dual equivalence} of
permutations: $\pi$ and $\tau$ are dual equivalent if one can be
obtained from the other by swaps of adjacent \emph{values} from the above
$P_{K}$, i.e., if their inverses are Knuth-equivalent, or
(equivalently) if they map to
the same second coordinate ($Q$-tableau) under the Robinson-Schensted
correspondence~\cite{HaimanDEq}. For the enumerative problems in this
paper, we get the same answers for Knuth and dual equivalence.  

In her dissertation~\cite{AssafDEG} Sami Assaf constructed graphs (with
some extra structure) whose vertices are tableaux of a fixed shape
(which may be viewed as permutations via their ``reading words''), and
whose edges represent (elementary) dual equivalences between
vertices.  For this particular relation (equivalently for the Knuth
relations), she was able to characterise the local structure of these
graphs, which she later used to give a combinatorial formula for the Schur
expansion of LLT polynomials and MacDonald Polynomials.  She also used
these, along with crystal graphs, to give a combinatorial realization of
Schur-Weyl duality~\cite{AssafSWD}.  

Sergey Fomin has a very clear elementary exposition of how Knuth and dual
equivalence are related to the Robinson-Schensted correspondence,
Sch\"utzenberger's \textit{jeu de taquin}, and the Littlewood-Richardson
rule in~\cite[Ch.~7, App.~1]{StanleyEC2}.  For the problems considered
above, the answers for $P_{K}^{\twolines}$ are well known to be: (A) the
number of involutions in $S_{n}$; (B) 1; and (C) $\{\iota_{n}\}$.  In fact one
can compute $\#\Eq^{\twolines}(\pi ,P_{K})$ for any permutation $\pi$ by
using the Frame-Robinson-Thrall hook-length formula to compute the
number of standard Young tableaux of the \emph{shape} output by the
Robinson-Schensted correspondence applied to $\pi$.

Any of the relations we consider can be naturally generalized to operate
on $\mathcal{W}([n])$, the set of \textbf{words} (with repeated entries
allowed) on the alphabet $[n]$: for example, the relation $123
\leftrightarrow 321$ would imply also moves of the form
$112\leftrightarrow 211$ and $122\leftrightarrow 221$, treating letters
with the same label within a word as increasing from left to right.  In
the case of the Knuth relations, the equivalence classes are simply the
\emph{elements} of the well-known \textit{plactic monoid} of Lascoux and
Sch\"utzenberger: $\mathcal{W}([n])/P_{K}$ \cite{LSPlactic, LLTPlactic}.
In \cite{Chinese}
the authors study the analogous
\textit{Chinese monoid}, which is $\mathcal{W}([n])/P_{3}$ (up to
the involution that reverses all words), for which they develop an
analogue of the Robinson-Schensted correspondence and count some of the
equivalence classes.  

Given that the Knuth relations act on adjacent elements, and lead to
some deep combinatorial results, it is perhaps not surprising that the
most interesting problems and proofs in this paper are to be found in 
Section~\ref{sec:adj}.  A summary of our numbers and results is given
in Figure~\ref{fig:sum}. 

An extended abstract of this paper appeared in the proceedings of
FPSAC10~\cite{LPRW}.  The third author is grateful to Sami Assaf, Karen
Edwards and Stephen Pon for helpful conversations. 

\begin{figure}[h]

\caption{Summary of Results}
\label{fig:sum}
\medskip 

\begin{minipage}[t]{\textwidth}
These tables give numerical values and names (when available) of the
sequences associated with enumerative questions (A) and (B).  All
sequences begin with the value for $n=3$.  Results proven in this paper
have a gray background; for other cases we lack even conjectural
formulae.  Six-digit codes preceded by ``A'' cite specific sequences in
Sloane~\cite{OEIS}.
\end{minipage}

\medskip

\begin{tiny} 

\begin{large}Number of classes\end{large}
\bigskip

\begin{tabular}{|c|l|l|l|l|}
\hline
\multicolumn{2}{|l|}{\multirow{2}{*}{\normalsize Transpositions}} & \small \S~2 & \small \S~3& \small\S~4 indices and\\
\multicolumn{2}{|l|}{} &{\small general} & {\small only indices adjacent} & {\small values adjacent} \\
\hline
(1) & $123 \leftrightarrow 132$ & \cellcolor{lightgray}[5, 14, 42, 132, 429] &\multirow{2}{*}{[5, 16, 62, 284, 1507, 9104]} & \multirow{2}{*}{[5, 20, 102, 626, 4458, 36144]} \\
\cline{1-2}
(2) & $123 \leftrightarrow 213$ & \cellcolor{lightgray}Catalan & & \\
\hline
\multirow{2}{*}{(4)} & \multirow{2}{*}{$123 \leftrightarrow 321$} & \cellcolor{lightgray}[5, 10, 3, 1, 1, 1] & \multirow{2}{*}{[5, 16, 60, 260, 1260, 67442]        } & \multirow{2}{*}{[5, 20, 102, 626, 4458, 36144]} \\
& & \cellcolor{lightgray}trivial & & \\
\hline
\multirow{2}{*}{(3)} & \multirow{2}{*}{$123 \leftrightarrow 132 \leftrightarrow 213$} & \cellcolor{lightgray}[4, 8, 16, 32, 64, 128] & \cellcolor{lightgray}[4, 10, 26, 76, 232, 764] & \multirow{2}{*}{[4, 17, 89, 556, 4011, 32843]} \\
& & \cellcolor{lightgray}powers of 2 &\cellcolor{lightgray}involutions  &\\
\hline
(5) & $123 \leftrightarrow 132 \leftrightarrow 321$ & \cellcolor{lightgray}[4, 2, 1, 1, 1, 1] & \multirow{2}{*}{[4, 8, 14, 27, 68, 159, 496]} & \multirow{2}{*}{[4, 16, 84, 536, 3912, 32256]} \\
\cline{1-2}
(6) & $123 \leftrightarrow 213 \leftrightarrow 321$ & \cellcolor{lightgray}trivial & & \\
\hline
\multirow{2}{*}{(7)} & $123 \leftrightarrow 132$ & \cellcolor{lightgray}[3, 2, 1, 1, 1, 1] & \multirow{2}{*}{[3, 4, 5, 8, 11, 20, 29, 57]} &  \multirow{2}{*}{[3, 13, 71, 470, 3497]} \\
& $\leftrightarrow 213 \leftrightarrow 321$ &  \cellcolor{lightgray}trivial & &\\
\hline
\end{tabular}

\vspace{3mm}

\begin{large}Size of class containing identity\end{large}
\bigskip

\begin{tabular}{|c|l|l|l|l|}
\hline
\multicolumn{2}{|l|}{\multirow{2}{*}{\normalsize Transpositions}} & \small \S~2 & \small \S~3& \small\S~4 indices and\\
\multicolumn{2}{|l|}{} &{\small general} & {\small only indices adjacent} & {\small values adjacent} \\
\hline
(1) & $123 \leftrightarrow 132$ &\cellcolor{lightgray}[2, 6, 24, 120, 720] &\cellcolor{lightgray}[2, 4, 12, 36, 144, 576, 2880] & \cellcolor{lightgray}[2, 3, 5, 8, 13, 21, 34, 55] \\
\cline{1-2}
(2) & $123 \leftrightarrow 213$ & \cellcolor{lightgray}$(n-1)!$ &\cellcolor{lightgray}product of two factorials & \cellcolor{lightgray}Fibonacci numbers \\
\hline
\multirow{2}{*}{(4)} & \multirow{2}{*}{$123 \leftrightarrow 321$} &\cellcolor{lightgray}[2, 4, 24, 720] &\cellcolor{lightgray}[2, 3, 6, 10, 20, 35, 70, 126] & \cellcolor{lightgray}[2, 3, 4, 6, 9, 13, 19, 28] \\
& &\cellcolor{lightgray}trivial & \cellcolor{lightgray}central binomial coefficients & \cellcolor{lightgray}A000930 \\
\hline
\multirow{2}{*}{(3)} & \multirow{2}{*}{$123 \leftrightarrow 132 \leftrightarrow 213$} &\cellcolor{lightgray}[3, 13, 71, 461] & [3, 7, 35, 135, 945, 5193] &\cellcolor{lightgray}[3, 4, 8, 12, 21, 33, 55, 88] \\
& &\cellcolor{lightgray}connected A003319 & Chinese Monoid~\cite{Chinese} &\cellcolor{lightgray}A052952 \\

\hline
(5) & $123 \leftrightarrow 132 \leftrightarrow 321$ &\cellcolor{lightgray}[3, 23, 120, 720]
 &\cellcolor{lightgray}[3, 9, 54, 285, 2160, 15825] &\cellcolor{lightgray}[3, 5, 9, 17, 31, 57, 105, 193] \\
\cline{1-2}
(6) & $123 \leftrightarrow 213 \leftrightarrow 321$ &\cellcolor{lightgray}trivial &\cellcolor{lightgray}separate formulae for odd/even &\cellcolor{lightgray}tribonacci numbers A000213  \\
\hline
\multirow{2}{*}{(7)} & $123 \leftrightarrow 132$ & \cellcolor{lightgray}[3, 23, 120, 720] &[4, 21, 116, 713, 5030] &\cellcolor{lightgray}[4, 6, 13, 23, 44, 80, 149, 273] \\
& $\leftrightarrow 213 \leftrightarrow 321$ & \cellcolor{lightgray}trivial &
&\cellcolor{lightgray}tribonacci A000073 $-[n\ \mathrm{ even}]$
\\
\hline
\end{tabular}
\end{tiny}
\end{figure}

If $\tau \in \Eq^\ourstar(\pi,P)$ we will say that $\tau$ is \emph{reachable}
from $\pi$ (under $P$). If $\Eq^\ourstar(\iota_n,P) = S_n$, then every 
permutation in $S_n$ is reachable from every other, and we will say that 
$S_n$ is \emph{connected} by $P$. If $\Eq^\ourstar(\pi,P) =\{\pi\}$ 
we will say that $\pi$ is \emph{isolated} (under $P$).

It is obvious that if $P_i$ refines $P_j$ as partitions of $S_{k}$
(i.e., $P_{i}\leq P_{j}$ in the lattice of partitions of $S_{k}$),
then the partition of $S_n$  induced by $P_i$ refines the one induced by
$P_j$, because a permutation reachable from $\pi$ under $P_i$ is also
reachable under $P_j$. This enables the following simple observations:

\begin{propo}\label{propo:Refine}
If $P_i$ refines $P_j$ (as partitions of $S_{k}$), then for all $\pi \in
S_{n}$ with $n\geq k$
\begin{eqnarray*}
\Eq^\ourstar(\pi,P_i) \subseteq \Eq^\ourstar(\pi,P_j) \\
\#\Eq^\ourstar(\pi,P_i) \leq \#\Eq^\ourstar(\pi,P_j) \\
\Num^\ourstar(n,P_i) \geq \Num(n,P_j)
\end{eqnarray*}
\end{propo}

\section{General pattern equivalence}\label{sec:gen}

In this section, we allow moves within an equivalence relation with no
adjacency restrictions.  This case is closely related to the theory of
\textit{pattern avoidance\/} in permutations: replacing one
pattern with another repeatedly leads eventually to a permutation which
avoids the first pattern.  

Some of the equivalence relations in this section are trivial, following
immediately from the following observation.  The others lead to familar
combinatorial numbers and objects.

\begin{propo}\label{propo:OneClass}
Fix $k$ with $2\leq k\leq n-1$, and let $P$ be any partition of
$S_{k}$. 
If $\Num^{\fourdots}(n-1,P) = 1$, then $\Num^{\fourdots}(n,P) = 1$.
\end{propo}
\begin{proof}
We will show that any $\pi \in S_n$ can be reached from the identity, $\iota_{n}$,
under the supposition that any two permutations in $S_{n-1}$ are
equivalent, in two stages.  If $\pi(1) \neq n$, simply apply the
supposition to the elements/positions $1 \dots n-1$ in $\iota_{n}$ to
obtain any permutation beginning with $\pi(1)$; if $\pi(1)=n $, we use
instead the elements/positions $1,3,4,5,\dotsc n$ (omitting 2, which is
$\leq n-1$ by hypothesis) to move $\pi (1)=n$ to the front of a
permutation equivalent to $\iota_{n}$.  Then in stage 2 we apply the
supposition to the elements now occupying \textit{positions\/} $2,
\dots, n$ to complete the construction of $\pi$.
\end{proof}

The following results follow.

\begin{propo}\label{propo:P457OneClass}
$\Num^\fourdots\left(n,\big\{ \{123,321\} \big\}\right) = 1$ for $n \geq
6$.  While for $n \geq 5$, we have $\Num^\fourdots\left(n,\big\{ \{123,132,321\} \big\}\right) = 1$
and $\Num^\fourdots\left(n,\big\{ \{123,132,213,321\} \big\}\right) = 1$ 

\end{propo}
\begin{proof}
It is easy to verify by hand, or by computer, that all permutations
in $S_5$ are reachable from $12345$ by moves in $P_5 = \big\{ \{123,132,321\} \big\}$.
(Indeed, all permutations in $S_4$ are reachable from $1234$ except for
$3412$, which is isolated.) As $S_5$ is connected, it follows 
(by induction) from the preceding proposition that $S_n$ is connected 
for all $n\geq 5$.

Proposition~\ref{propo:Refine} tells us that $S_n$ is connected under
$P_7 = \big\{ \{123,132,213,321\} \big\}$ whenever it is connected under
$P_5$ since $P_7 \geq P_5$.  (In $S_4$, the permutation $3412$ remains
isolated.)
Finally, we can check by computer that under $P_4 = \big\{ \{123,321\} \big\}$
$S_6$ is connected; whence, $S_n$ is connected for $n \geq 6$. 
\end{proof}

We remark that under $P_4$, $S_4$ splits into 10 equivalence classes,
and $S_5$ into three classes. The class containing $12345$ contains 24
elements.  This suggests a possible bar bet.  Hand your mark six cards
numbered $1$ through $6$ and invite him or her to lay them out in any
sequence. By applying moves of the form $123 \leftrightarrow 321$
(``Interchange two cards if and only if an intermediate (value) card
lies (in any position) between them.'')  you will always be able to put
the cards in order.  (It may take some practice, however, to become
proficient at doing this quickly.)  Now ``go easy'' on your mark by
reducing the number of cards to 5. Even from a random sequence, the mark
has only one chance in five of being able to reach the identity.

Of course from Proposition~\ref{propo:P457OneClass} it immediately follows that:
\begin{cor}\label{cor:P457Id}
$\#\Eq^\fourdots\left(\iota_n,\big\{ \{123,132,321\} \big\}\right) =
n!$, 
$\#\Eq^\fourdots\left(\iota_n,\big\{ \{123,132,213,321\} \big\}\right) = n!$ for $n \geq 5$; and
$\#\Eq^\fourdots\left(\iota_n,\big\{ \{123,321\} \big\}\right) = n!$ for $n \geq 6$.
\end{cor}

\begin{propo}\label{propo:P2Id}

$\#\Eq^\fourdots\left(\iota_n,\big\{ \{123,213\} \big\}\right) = (n-1)!$ for $n \geq 2$.
\end{propo}
\begin{proof}
Obviously the largest element $n$ cannot be moved away from the end
of the permutation. Equally obviously the $n$, remaining at the far right,
enables the other elements to be freely pairwise transposed, thereby
generating any permutation in $S_{n-1}$.  

\end{proof}

\begin{propo}\label{propo:P2ClassCat}

For $n \geq 1$, $\Num^\fourdots\left(n,\big\{ \{123,213\} \big\}\right) = c_n =
\frac{2n!}{n!(n+1)!}$, the $n$th Catalan number.
\end{propo}
\begin{proof}
Let $\pi \in S_{n}$.  If $i < j < k$, and $\pi(i) < \pi(j) < \pi(k)$,
then $\pi(k)$ enables the swapping 
of $\pi(i)$ and $\pi(j)$ to arrive at a permutation $\pi^{(1)}$ with a strictly larger
number of inversions. We can continue in this way to obtain a sequence
$\pi =\pi^{(0)}\rightarrow \pi^{(1)}\rightarrow \dots \rightarrow \pi^{(n)}$,
where $\pi^{(n)}$ has no such triples, i.e., $\pi^{(n)}$ is $123$-avoiding.
It remains to show that no matter which sequence of moves we make, the
final permutation $\pi^{(n)}$ is unique.  

Call an element in a permutation $\sigma \in S_{n}$, a
\textit{right-to-left maximum} if it is greater than every element that
occurs to its right. (More formerly, $\sigma_{k}$ is a right-to-left
maximum if $\sigma_{k}>\sigma_{l}$ for all $k<l\leq n$.)  The set
$\mathcal{M}(\pi )$ of these elements remains fixed under the relation
above, and forms a decreasing subsequence of $\pi^{(i)}$ for all $i$.  Now
a permutation is $123$-avoiding if and only if it is a union of at most
two decreasing subsequences.  So $\pi^{(n)}$ must be the unique
permutation formed by rearranging the elements of $\pi$ not in
$\mathcal{M}(\pi )$ in decreasing order.  It is clear that the elements
of $\mathcal{M}(\pi )$ are positioned so as to enable all the necessary
transpositions.  

Thus the ``largest'' (by number of inversions) elements in each equivalence
class are exactly the $123$-avoiding permutations, of which there are
$c_n$~\cite[ch.~14]{BonaWalk} or~\cite[Sec. 4.2]{BonaPerm}.  (Similarly
one can show that the ``smallest'' elements are the $213$-avoiding
permutations.) 
\end{proof}

\begin{eg}
\label{eg:P2ClassCat}
Working within $S_{9}$, we have the following sequence of equivalences,
where elements about to be transformed are indicated in \textbf{bold}.  
The subsequence of right-to-left maxima is $976$.  
\[
\mathbf{3}82941\mathbf{57}6 \rightarrow 58\mathbf{2}9\mathbf{4}13\mathbf{7}6\rightarrow 
5849\mathbf{2}1\mathbf{37}6\rightarrow 58493\mathbf{127}6 \rightarrow 
\mathbf{58}4\mathbf{9}32176 \rightarrow 854\mathbf{9}321\mathbf{76} 
\]
The reader is encouraged to draw corresponding permutation matrices or
diagrams, which clarify visually how the right-to-left maxima facilitate
the transformation of the other elements into a decreasing subsequence.  
\end{eg}

The next two propositions study an equivalence relation and class whose
enumeration is equivalent under symmetry (reversal or complementation)
to $\Eq^{\fourdots}(\iota_n,P_3)$.  
The first leads to \textit{connected}
or \textit{indecomposable permutations}~\cite[A003319]{OEIS}, namely those not fixing
$\{1,2,\dotsc j\}$ for any $1\leq j <n$.  If we define the
\textit{direct sum} of two permutations so that it corresponds to the
direct sum of the corresponding permutation matrices, then these are
simply the permutations which are indecomposable as direct sums in the
usual matrix sense.  Some authors use the term
\textit{plus-indecomposable}~\cite{AAKSimple} to describe this class.  
The second leads to the \textit{layered permutations}, namely those which are a
direct sum of decreasing permutations, introduced by
W.~Stromquist~\cite{StromLayered}, and studied carefully by A.~Price in
his thesis~\cite{PriceLayered}.  

\begin{propo}\label{propo:P3revIndec}
Let $\rho_{n}$ denote the ``reverse word'' permutation $n,n-1,\dots 1$.
Then $\Eq^{\fourdots}(\rho_{n},\big\{ \{321,312,231\} \big\})$ is the set of indecomposable 
permutations.  
\end{propo}

\begin{proof}
When viewed as a $(0,1)$-matrix, any permutation decomposes as a direct sum of
irreducible blocks along the main diagonal; in particular, the identity
$\iota_n$ decomposes into $n$ singleton blocks, while $\rho_{n}$ is
indecomposable and is one large block. A permutation is connected if and
only if it decomposes as a single block.  

First note that if any transformation of entries $(a_1,a_2,a_3) \rightarrow
(b_1,b_2,b_3)$ applied within a block causes it to split into more than
one block, then $b_1$ must be in the leftmost/lowest of the new blocks,
and $b_3$ in the rightmost/highest. Therefore $b_1$ must be less than
$b_3$, which is exactly what does not happen with any of our possible
transformations, because the first element is larger than the third in
each of 321, 312 and 231.  Since all of our transformations are
reversible, this shows also that we cannot combine blocks. Thus,
\textit{the irreducible block structure of a permutation does not change
under these transformations.}  In particular, if we start with an
indecomposable permutation such as $\rho_{n}$, successive applications
of the permitted operations will always produce indecomposable
permutations.

Next we have to show that all indecomposable permutations are in
fact reachable from $\rho_{n}$. Remembering that our replacement operations are
all reversible, we will instead show that we can always return to $\rho_{n}$
from an arbitrary indecomposable permutation.
Take $n\geq 3$, and let $\tau  = \tau_1\tau_2\dotsb \tau_n$ be an arbitrary 
indecomposable permutation other than $\rho_{n}$. We will show that $\tau $ 
always contains at least one of 312 or 231. It's easy to see that $\tau$
must have an \textit{ascent}, i.e., there exists $k$ such that
$\tau_{k}<\tau_{k+1}$.  
Now if any element to the right of $\tau_{k+1}$ is less than $\tau_k$ we
have a 231, so assume there are none such. Similarly, assume there is no
element to the left of $\tau_k$ and greater than $\tau_{k+1}$ (avoiding
312).
But there must be some $y$ to the left of $\tau_k$ which is greater than
some $x$ to the right of $\tau_{k+1}$, or otherwise the permutation decomposes
between $\tau_k$ and $\tau_{k+1}$. These four elements $y,\tau_k,\tau_{k+1},x$ form a
3142, which contains both a 312 $(y,\tau_k,x)$ and a 231 $(y,\tau_{k+1},x)$.

Having now located a 312 or 231, we can then apply either $312 \rightarrow 321$
or $231 \rightarrow 321$, as appropriate. Each of these operations 
simply switches
a pair of elements, and (as we have seen in the proof of
Proposition~\ref{propo:P2ClassCat}) strictly increases 
the number of inversions, progressing us toward $\rho_{n}$. 
This completes the proof that all indecomposable permutations are 
reachable, and therefore the proof that the reachable permutations are
exactly the indecomposable permutations.
\end{proof}

\begin{propo}\label{propo:P3revNum}
$\Num^\fourdots\left(n,\big\{ \{321,312,231\} \big\}\right) = 2^{n-1}$ for $n \geq 1$.
\end{propo}
\begin{proof}
As we saw in the proof of the previous proposition, the irreducible
block structure of a permutation does not change under the
transformations we are considering here.  By the arguments already
given, we can work within any indecomposable block to restore it to an
anti-identity. Therefore each equivalence class consists of all the
permutations with a given block structure, and contains exactly one
permutation which is a direct sum of anti-identities.

These are exactly the layered permutations, and there are clearly
$2^{n-1}$ of them, with a factor of $2$ according to whether each consecutive
pair of elements is or is not in the same layer. (Equivalently, any such
permutation is determined by the composition of $n$ representing its
block sizes, of which there are $2^{n-1}$.)  
\end{proof}

Finally we apply the reversal (or complementation) involution on $S_{n}$
to the above result to get our result for the partition $P_{3}$.  

\begin{theorem}\label{thm:P3Id}
$\Num^\fourdots\left(n,\big\{ \{123,132,213\} \big\}\right) = 2^{n-1}$ for $n \geq 1$.
\end{theorem}

\section{Adjacent transformations}\label{sec:adj}

As mentioned in the introduction, this section contains our most
interesting results and proofs.  The first rediscovers sequence A010551
from Sloane~\cite{OEIS}. 

\begin{theorem}\label{thm:P2bId}
$\#\Eq^\twolines\left(\iota_n,\big\{ \{123,213\} \big\}\right) = \floor{n/2}!\ceiling{n/2}!$ 
for $n\geq 1$.
\end{theorem}
\begin{proof}
Generically stated, our rules in this case allow the transposition of
any two adjacent elements if the element immediately to their right is
bigger than both of them.  Applying these successively to $\iota_{n}$,
we note first that the largest element, $n$, never comes unglued from
the right end, because there is nothing to enable it; therefore, $n-1$
must stay somewhere in the last three positions (as only $n$ can enable
its movement).  Similarly, $n-2$ remains somewhere in the last five,
$n-3$ within the last seven and so on; such restrictions apply to the
largest $\floor{n/2}$ of the elements.  This limits the number of
permutations potentially reachable to $\floor{n/2}!\ceiling{n/2}!$: placing
the elements from largest to smallest, one has a choice of
$1,2,3,\ldots,\floor{n/2},\ceiling{n/2},\ldots,3,2,1$ positions to put each
element.

Next we will show that all permutations conforming to
these restrictions are indeed reachable from $\iota_{n}$. We will do this in 
two stages. In Stage~1 we move each of the large, constrained
elements as far left as it can go. (In the most natural way to 
achieve this, the smaller, unconstrained elements remain in
their natural increasing order, although we shall see that this
does not matter as they can then be permuted freely.) In Stage~2
we construct the target permutation two elements at a time,
working from left to right.

\textit{Stage~1: Maximally spread out the $\floor{n/2}$ largest
elements. \/}  First move $\floor{n/2}+1$ one step left, using a move of
type $123 \rightarrow 213$, in which $\floor{n/2}+2$ plays the role of
the facilitating ``3''. In the same way, move the element
$\floor{n/2} + 2$ to the left, continuing until the entire block
$\floor{n/2},\dots,n-1$ has been shifted one to the left. The element
$n-1$ has now reached its leftmost permitted position, and will remain
in place as we now similarly transform the block $\floor{n/2},\dots,n-2$.
This moves $n-2$ as far left as it will go, and we now transform the next
smaller block, etc. Continue until reaching a permutation which alternates the
subsequences $1,2,\dotsc \floor{n/2}$ and $\floor{n/2}+1,\dotsc ,n$ (e.g.,
$15263748\in S_{8}$ or $516273849\in S_{9}$).
This places each constrained element (in the latter subsequence) as far
left as possible. These elements will now serve as a ``skeleton''
enabling the construction of the target permutation.

\textit{Stage~2: Construct the target permutation. \/} The key observation
making this stage possible is that 
the small, unconstrained elements can be freely moved about while leaving
the large elements in the skeleton fixed. For if
$\{a,b\} < X < Y$, we can always execute the following sequence of
moves: $a\mathbf{XbY} \rightarrow \mathbf{abX}Y \rightarrow
b\mathbf{aXY} \rightarrow bXaY$. 
In the case where $n$ is odd, we may consider the leftmost element in the
skeleton to be in position 3, and the two small elements in positions
1 and 2 can be interchanged if desired.

Now we examine the target permutation and move the required element(s)
into the first position (if $n$ is even), or the first two positions
(if $n$ is odd). At this point, the elements occupying the next two
positions are reclassified as small, so that the skeleton terminates
two positions further to the right, and we continue by placing and
ordering the next pair of elements. By continuing two elements at a time, we
can build the entire target permutation. 
\end{proof}

\begin{eg}
To reach the target permutation $452637819$ according to the above
scheme we would apply the following moves.  The numbers indicated in
\textbf{bold} are about to be transposed, either by a standard
$\mathbf{12}3\rightarrow \mathbf{21}3$ move, or by the move
(suppressing intermediate steps) $ \mathbf{a}X\mathbf{b}Y\rightarrow 
\mathbf{b}X\mathbf{a}Y$ described
above.  

\[
\begin{aligned}     
123\mathbf{45}6789 &\rightarrow 1235\mathbf{46}789 &\rightarrow 12356\mathbf{47}89 &\rightarrow  123567\mathbf{48}9 &\rightarrow 12\mathbf{35}67849 &\rightarrow 125\mathbf{36}7849 \\  
&\rightarrow 1256\mathbf{37}849 &\rightarrow 1\mathbf{25}673849 &\rightarrow 15\mathbf{26}73849 &\rightarrow \mathbf{15}6273849 &\rightarrow \cred{5}1\cred{6}2\cred{7}3\cred{8}4\cred{9} \\  
51627\mathbf{3}8\mathbf{4}9 &\rightarrow 516\mathbf{2}7\mathbf{4}839 &\rightarrow 5\mathbf{1}6\mathbf{4}72839 &\rightarrow \mathbf{54}6172839 &\rightarrow 456\mathbf{1}7\mathbf{2}839\\ 
&\rightarrow 45\mathbf{62}71839 &\rightarrow 45267\mathbf{1}8\mathbf{3}9 &\rightarrow 4526\mathbf{73}819 &\rightarrow 452637819 
\end{aligned}
\]

\end{eg}

\begin{theorem}\label{thm:P5bId}
Let $n$ be an integer $\geq 3$, and for any odd positive integer $m$ set
$m!! = 1\cdot 3\cdot \dotsb \cdot m$, the product of odd natural numbers
less than or equal to $m$.  Then 

\[
\#\Eq^\twolines\left(\iota_n,\big\{ \{123,132,321\} \big\}\right) =
\mycases{ \frac{3}{2}(k)(k+1)(2k-1)! & \kern -6.2pt for $n=2k+1$ odd. \cr
\frac{3}{2}(k)(k-\frac{1}{3})(2k-2)! - (2k-3)!! & \kern -6.2pt for $n=2k$ even.\cr}
\]

\end{theorem}

\begin{proof}
As in the previous proof, we begin by giving a set of necessary conditions
for the a permutation to be reachable from $\iota_{n}$, then show how to reach each 
such permutation, thereby proving that our conditions in fact
characterise $Eq^\twolines\left(\iota_n \right)$.  

The first restriction is that the element 1 must occupy
a position of odd index, because it can only participate
in a move as a ``1'', and every move either leaves it 
fixed or moves it by two positions.
The second restriction is that the element 2 cannot occupy
a position of odd index to the left of 1, because if it
winds up to the left of 1, its last move there was
$132 \rightarrow 321$, and since 1 is always in a position of odd index,
this places 2 into a position of even index. Then it stays
on the left of the 1, so it must play the role of ``1'' in 
any future swaps, again preserving the parity of its position.
Let us call the class of permutations thus described $\mathcal{A}_n$, the
class of \emph{admissible} permutations.

Now in the case where $n=2k+1$ is odd, this characterization is exact,
so we will first complete the proof for odd $n$. In the case where $n=2k$
is even, there are a small number of exceptional permutations which must
be excluded; we will turn to these at the end of the proof.

\textit{Case~1: $n$ is odd. \/}  First we count the number of admissible permutations: 
If the 1 is in position 1, then the 2 can be in any of
$n-1$ positions, and the remaining $n-2$ elements can be arranged
in $(n-2)!$ ways.  
If the 1 is in position 3, then the 2 can be in any of 
$n-2$ positions; if the 1 is in position 5, then in any of 
$n-3$ positions, and so forth, while in each case, the remaining 
$n-2$ elements can be placed freely.
Summing over the possible locations for the element 1, we arrive
at the given formula for odd $n$, and also at the formula for even $n$
upon suppression of the double-factorial correction term. 
For example $\mathcal{A}_{5}$ consists of 54 permutations: 
all 24 of the form $1****$, 18 of the form $**1**$ (all but the six of
the form $2*1**$), and 12 of the form $****1$ (all but those of the form
$2***1$ or $**2*1$).  

It remains to show that all admissible
permutations are in fact reachable. We do this in two stages.

\textit{Stage~1:\/} First we will show that all permutations beginning with a 1 are
reachable from the identity. We proceed in steps; after each, we will have
a monotonically increasing initial segment, followed by 
a segment that matches the target permutation.  This segment gets
created from right to left, each step increasing the length of the
completed segment by 1 by selecting and moving one element from the
increasing segment to the left end of the completed segment.  

Note that within an increasing segment, the concatenation of moves
$a\mathbf{bxy}\rightarrow \mathbf{ayx}b\rightarrow axyb$ allows a
selected element $b>1$ to move two positions rightward while maintaining
that the segment to its left is increasing.  So if the target position
for $b$ is an even number of positions away, an appropriate number of
such moves will suffice.  If $b$ is an odd number of positions away,
first apply the move $\mathbf{abx}y\rightarrow axby$, then proceed as
before.  This shows that we can reach any permutation that starts with a
1 from $\iota_{n}$.  

\textit{Stage~2:\/} To show that we can get to the identity from an
arbitrary admissible permutation, it remains to show that the element 1
can always be moved to the front of such a permutation. In fact we
only need show that the element 1 can always be moved \emph{toward} the
front (necessarily by two positions), and then we can just move it
repeatedly until it is \emph{at} the front.

If the 1 is at the very end of the permutation, the 2 must be to its
left and in a position of opposite parity.  Move the 2 rightward using
moves $123\rightarrow 321$ or $132\rightarrow 321$ (the 2 functioning as
a ``1'') until it is adjacent to the 1; the 1 can then be moved
leftward.

We use the term \textit{$k$-factor} here and in later proofs as
shorthand for ``length $k$ subsequence of adjacent elements'' of a
permutation. If the 1 is not at the very end of the permutation, then
consider the 5-factor centred on the 1. (At this point, we
are relying on the assumption that $n$ is odd, because the largest
element occupies a position of odd index, and therefore if it is
not at the end of an odd permutation, it must be at least two
positions away from the end, guaranteeing the 5-factor
which we need. We will return to this point when we consider the
even case below.) There are 24 cases. 
For 18 of these cases, we know that we can convert this segment
to an increasing one (or to any other permutation beginning
with a 1) using the analysis for $n=5$, which is easy to check by hand.  
The cases which cannot be handled are those of the form 2*1**.
We will add a preprocessing step to make sure that we are not in
such a case. Namely, we will locate the element 2 (the actual 2)
and move into one of the spaces indicated by a $*$.  

If the 2 is somewhere to the left of the 1 then the same argument used
above in the case of permutations ending in 1 again shows that the 1 can
be moved leftward.  

If the 2 is somewhere to the right of the 1, we will go and fetch it as
follows. Move the 1 to the \emph{right} until it is either one or two
positions left of the 2. We do this by moving it two positions at a
time, using either $123 \rightarrow 321$ or $132 \rightarrow 321$, as
required.  This leaves behind a consecutive trail of elements in which
each odd-position element is larger than the even-position element which
follows it.  We will call these ``odd/even descents''.

If the 2 was in an odd position, we will arrive at $1x2$, which we
correct to $12x$. If the 2 was in an even position, we will arrive
at $12x$ directly. 

Now we pull both the 1 and the 2 back through the odd/even descents by a
sequence of consecutive moves of the form $\mathbf{yx1}2\rightarrow
\mathbf{1xy}2\rightarrow 1\mathbf{yx2}\rightarrow 12yx$, where $y>x>2$.
(We may also apply $12\mathbf{yx}\rightarrow 12xy$ if we wish, but this
isn't necessary.)  This brings us to a permutation in the same
equivalence class, where the 5-factor has been modified to **12*. But
we know that just working within this 5-factor, we can use the $n=5$
case to modify it to the form $12345$ (where the 1 and 2 are actual
values, the others relative values).  In particular, we have moved the
actual 1 two spaces to the left.  Doing this repeatedly gets us to a
permutation beginning with 1, which we have seen in Stage~1 is
equivalent to $\iota_{n}$.

\textit{Case~2: $n=2k$ is even. \/}  
In the even case, we need to describe an additional class of
permutations that not reachable from $\iota_{n}$.  Let $\mathcal{X}_{n}$
consist of all permutations obtainable as follows: Fill the positions in
order $n-1,n,n-3,n-2,n-5,n-4 \ldots 3,4,1,2$, according to the following  
rule. When filling positions of odd index, the smallest available element 
must be chosen; the subsequent selection of an element to place to its 
right is then unconstrained. Thus 1 must be placed in position $n-1$,
and the element placed in position $n$ could be any other number; however, if it
is not 2, then the 2 is immediately placed in position $n-3$; otherwise
$3$ is placed in this position.   
For example, $\mathcal{X}_{4}=\{3412, 2413, 2314\}$ and 
\[
\mathcal{X}_{6}=\left\{
\begin{matrix}
563412 & 562413 & 562314 & 462315 & 452316 \\
463512 & 462513 & 362514 & 362415 & 352416 \\
453612 & 452613 & 352614 & 342615 & 432516 \\
\end{matrix}
 \right\}\,.
\]

The number of permutations in the class $\mathcal{X}_n$ just described is $(n-1)!!$. 
As we will see next, none of them is reachable. However, it is also true
that most of them are not in $\mathcal{A}_n$, and therefore have not been 
included in the enumeration; this is because most of
the permutations in $\mathcal{X}_n$ have the $2$ in position $n-3$, which is a
position of odd index to the left of the $1$. The only permutations in
$\mathcal{X}_n$ which we have counted, and which therefore must be subtracted off,
are the ones where the $2$ is in position $n$, of which there are $(n-3)!!$.

To see that none of the permutations in $\mathcal{X}_n$ is reachable, consider
their $3$-factors. Every $3$-factor centred on a position of odd index
is either a $213$ or a $312$, because the middle element was placed before
either of its neighbours, and was the minimal available element at the 
time it was placed. And every $3$-factor centred on a position of even
index is a $231$, because the elements in positions of odd index, which
are the minimal elements, descend from left to right. Therefore permutations
belonging to $\mathcal{X}_n$ contain \emph{no} factors of form $123$, $132$, or
$321$, and are therefore isolated by the relation, each one being a singleton 
equivalence class. In particular they are not in the equivalence
class of the identity.

Now we have to consider which permutations in $\mathcal{A}_n$ are not in fact
reachable. The proof for odd $n$ almost carries through completely; indeed,
as remarked, it only fails when the element 1 lies in the penultimate
position $n-1$. We have already seen that the permutations belonging
to $\mathcal{X}_n \cap \mathcal{A}_n$ are not reachable; we will show that all
others are. Take any permutation $\pi \not\in \mathcal{X}_n$, but with the
minimal element $1$ placed in position $n-1$. Checking the conditions
from right to left, suppose the element $\pi_{j}=y$ represents
the last time that we were in compliance with the conditions, and suppose
$\pi_{i}=x$ is the first minimal element which has not 
gone where it should go. That is, all odd positions from $j$ to $n-1$ are
occupied by elements which are left-to-right minima, but the smallest
element situtated in positions $1$ through $j-1$ is not in position
$j-2$, as expected, but in position $i$ with value $x$.

As before, all we need to do is show that we can move the element $1$
to the left. This exploits two facts: that $x$ is the minimal element in
a lefthand region, and the righthand region is alternating.

Because the righthand portion of $\pi$, from position $j$ onwards, is
alternating, with every step from an odd to an even position being an
ascent, and every step from even to odd being a descent, we will have
a particular interest in a certain type of $3$-factor beginning in
a position of odd index. Namely, we will refer to a $3$-factor
$\pi_{h},\pi_{h+1},\pi_{h+2}$ as an \emph{odd 321} if 
$\pi_{h}>\pi_{h+1}>\pi_{h+2}$ and if $h$ is odd. Note that an odd $321$
beginning in position $n-3$ is exactly what we need, because either option
for replacing it shifts the element $1$ from position $n-1$. 

First, take the element $x$ and use moves $\rightarrow 321$ to shift 
it rightward, two positions at
a time, until it arrives in position $j-2$ or $j-1$. This is possible
because $x$ is moving through a region in which it is itself the minimal
element.

Now $j-2$ is an odd position, so if $x$ has reached position $j-2$
then positions $j-4,j-3,j-2$ now form an odd $321$. Alternatively, if $x$
has reached position $j-1$ then positions $j-2,j-1,j$ now form an odd $321$,
because the second and third of these positions are occupied by $x$ and $y$
and $y<x$. We will show that we can propagate either of these odd $321$s 
rightward until they capture the smallest element, which can then be moved.

In either case, we have an odd $321$, followed, two positions later, by
an element which is smaller than everything to its left. This gives us,
in other words, a configuration $432-1$, which, filling in the blank,
might actually be (a) a $54231$, (b) a $53241$, or (c) a $43251$. Check that the
following moves are available in each case: 
(a) $54231 \rightarrow 24531$; (b) $53241 \rightarrow 23541$; 
(c) $43251 \rightarrow 23451 \rightarrow 25431$. 

Note that these moves each replace a configuration which begins with
an odd $321$ by one which \emph{ends} with an odd $321$. And, because
of the placement of the left-to-right minima, this new odd $321$ either
terminates with the smallest element $1$, or again has another left-to-right
minimum two positions to its right.

Therefore we can propagate the $321$ rightward until it reaches the smallest
element; therefore we can move the smallest element; therefore the 
permutations not belonging to $\mathcal{X}_n$ are in fact reachable.
  
This completes the missing step in the proof for even $n$. 
\end{proof} 

\begin{theorem}\label{propo:P4bId}

$\#\Eq^\twolines\left(\iota_n,\big\{ \{123,321\} \big\}\right) = \dbinom{n-1}{\lfloor (n-1)/2 \rfloor}$.

\end{theorem}
\begin{proof}
We claim that the permutations in this class are direct sums of singletons
and of blocks of odd size greater than one, where within each
block the even elements (with respect to the block) are on the
diagonal, and the odd elements form an \textit{indecomposable}
$321$-avoiding permutation. 

Let us call the set that we have just described $\mathcal{A}_n$.
Because all the even elements within a block are fixed points of the
permutation, the indecomposability of the odd elements is equivalent to
the indecomposability of the entire block.

First we will show that $\mathcal{A}_n$ is closed under $123
\leftrightarrow 321$; since the identity is in $\mathcal{A}_n$ this will
establish that the equivalence class of the identity is a subset of
$\mathcal{A}_n$. Then we will show that we can return to the identity
from any permutation in $\mathcal{A}_n$, which will establish that the
two sets are identical. Finally we will use generating functions to
enumerate $\mathcal{A}_n$.

Let $\pi$ be an arbitrary permutation belonging to $\mathcal{A}_n$. By
definition, $\pi$ is a direct sum of singleton blocks and of larger
blocks having a specific form. We will call any non-singleton block of
$\pi$ \emph{large}. Unless $\pi$ is the identity, it contains at least
one large block. Note that large blocks always begin with descents: for
if the first element of the block were on the diagonal, we could split
the block immediately after it to obtain a direct sum decomposition;
therefore, the first element is below the diagonal (i.e., is an
excedance) but the second element is on the diagonal. For symmetric
reasons, large blocks end with descents as well.

First we show that any application of $123 \rightarrow 321$ to $\pi$
produces an element $\pi '$ in $\mathcal{A}_n$. Consider the different ways that a 3-factor
$\pi_i,\pi_{i+1},\pi_{i+2}$ of form $123$ might occur within $\pi$.

\textit{Case~(a)} All three elements are in singleton blocks; then the result is the
unique large block of size 3 permissible within elements of
$\mathcal{A}_{n}$, namely 
$\begin{pmatrix}
0&0&1\cr 0&1&0\cr 1&0&0\cr 
\end{pmatrix} $.  

\textit{Case (b)} Exactly two of the elements (necessarily the first two
or the last two) are in singleton blocks. Assume without loss of
generality that it is the last 
two, i.e., $\pi_{i+1}=i+1$ and $\pi_{i+2}=i+2$, and that $\pi_i < i$
belongs to a large block $B$ of size $2k+1$. Since $\pi_{i}$ is the last
element of $B$, it must be an odd element within the
block.  The replacement produces a larger block $B'$ of size $2k+3$.
The $k$ even elements of $B$, along with $\pi_{i+1}$ are diagonal
elements that remain unchanged by the transformation, so all the even
elements of $B'$ lie on the diagonal.  The block $B'$ must be
indecomposable, because any breakpoint before $\pi'_{i}$ would already
have been a breakpoint for $B$ itself, and no breakpoint can occur
thereafter since $\pi'_{i}>\pi'_{i+1}>\pi'_{i+2}$.  

Finally, $\pi'_{i}$ is the largest element of the block, so could only
play the role of ``3'' in a 321 pattern, but there is only one odd
element to its right in the block.  So any 321 pattern of odd elements
in $B'$ that did not already exist in $B$ must use $\pi'_{i+2}$ as ``1''
and odd elements to the left of $\pi'_{i}$ for ``3'' and ``2''.  But
then these elements (which haven't moved) together with $\pi_{i}$ would
have formed a 321 in $B$, which wasn't allowed.

\textit{Case~(c)} Just one element is in a singleton block. 
This can't be the third (or, symmetrically, the first) element, because
if $\pi_i$ and $\pi_{i+1}$ are the final two elements of a large
block, then $\pi_i > \pi_{i+1}$, so our 3-factor is not a 123.  So it must be
$\pi_{i+1}$ which is the singleton, while the other two elements belong
to two large blocks. The replacement merges these three blocks into one;
the even elements, including $\pi_{i+1}$, remain on the diagonal, and as
in the previous case any point at which the new block split would also
imply a decomposition of one of the old blocks at the same
position. The odd elements of $\pi '$ are $321$-avoiding because if a $321$
contained just one of $\pi'_i$ or $\pi'_{i+2}$ then it would be
pre-existing (with $\pi_{i+2}$ or $\pi_{i}$ respectively). If it
contained both, then the third element in the pattern would be either on
the left and too large for the old lefthand block, or on the right and
too small for the old righthand block. 

\textit{Case (d)} The three elements are all within a single large block
$\mathcal{B}$. 

First we claim that the middle element, $\pi_{i+1}$ must be in an even
position (within $\mathcal{B}$).  Otherwise, $\pi_{i}$ and $\pi_{i+2}$
would be in even positions, hence on the diagonal, and the 123 form of
the 3-factor would mean $\pi_{i+1}$ was also on the diagonal; thus,
$\mathcal{B}$ would have to be of size at least 5.  Now if all elements to
the left of $\pi_{i}$ were smaller than it, $\mathcal{B}$ would split into
summands before position $i$.  But if some $\pi_{j}$ is greater than
$\pi_{i}$ (for some odd $j<i$) it
must be greater than $\pi_{i+2}$, forcing a compensatory $\pi_{k} < \pi_{i}$ for
some odd $k>i+2$.  But then $\pi_{j}$, $\pi_{i+1}$, $\pi_{k}$ formed a
321-pattern of odd positions within the block, contrary to hypothesis.
The claim follows.    

Now the replacement $\pi_{i}\pi_{i+1}\pi_{i+2}\rightarrow
\pi_{i+2}\pi_{i+1}\pi_{i}$ cannot create a new direct sum decomposition
since it is increasing the left element and decreasing the right one.
Suppose that somehow this move created a $321$ among the odd elements
(within $\mathcal{B}$).  If it only used one of $\pi_{i}, \pi_{i+2}$,
then it must have been pre-existing with the other one, contrary to
hypothesis.  If it used both, then without loss of generality assume
$\mathcal{B}$ contains an element $x$ to the left of the replaced
3-factor, but larger than $\pi_{i+2}$. Because $x$ is also greater than
$\pi_{i+1}$, it uses up one of the odd values greater than the diagonal
element $\pi_{i+1}$, meaning that there must be a $y$ to the right of
$\pi_{i+1}$ but smaller than it, and then $x,\pi_{i+2},y$ is a
pre-existing $321$.  The case where $\mathcal{B}$ contains an element
$y$ to the right of the replaced 3-factor, but smaller than $\pi_{i}$
follows by symmetry.

\textit{Non-Case (e)} The last possibility to consider is that the
3-factor is split across two adjacent large blocks, necessarily with two
elements at the start or end of one of the large blocks.  But this is
ruled out because large blocks begin and end with descents.

Note that in each of these cases the replacement $123
\rightarrow 321$ winds up gluing together all the blocks of $\pi$ which
it straddles, leaving the same number or fewer blocks in $\pi'$.  In
particular, the replacement may glue together blocks, but never splits
any apart.

Now consider applications of $321 \rightarrow 123$ within a permutation
$\rho \in \mathcal{A}_n$ to obtain a new permutation $\hat{\rho}$. Clearly,
any adjacent $321$ must lie within a single 
block, as in any two blocks, all the elements in the block to the right
are larger than all the elements in the block to the left. Because the
even elements within a block increase monotonically, the $321$ is
composed of odd, even, odd elements. An analysis similar to that given
above shows that any such 
transformation is simply the reverse of one of the cases (a--d)
described above, so $\hat{\rho}$ is always in $\mathcal{A}_{n}$.  

Now we need to show that we can use these transformations to return to
the identity from any permutation $\sigma$ in $\mathcal{A}_n$. We first
claim that every large block of $\sigma$ contains a $321$ as a factor.
For the first element of the block must lie below the diagonal and the
last element must lie above it; therefore, there are two consecutive odd
elements in the block with the first below and the second above the
diagonal. Together with the even element which separates them, and which
lies \emph{on} the diagonal, this forms a $321$.

Unless $\rho $ is itself the identity, it contains a large block, and
therefore a $321$. Replacing this $321$ with a $123$ yields a
permutation $\hat{\rho}$ having strictly fewer inversions than $\rho$. But
as $\mathcal{A}_n$ is closed under such replacements, we know that $\hat{\rho}$
also belongs to $\mathcal{A}_n$, and therefore is either the identity or
else contains a $321$. By iterating this process, we must eventually
arrive at a permutation having no inversions, namely the identity.

This establishes that $\Eq^\twolines\left(\iota_n,\big\{ \{123,321\}
\big\}\right)=\mathcal{A}_{n}$, so all that is left is the enumeration
of these classes.  It is an easy exercise
\cite[(n$^{6}$)]{StanleyCatAdd} or \cite[p.~15]{ClaKit} that the number of \emph{indecomposable} 
$321$-avoiding permutations on $m+1$ elements is the Catalan number
$c_{m}={\frac{1}{m+1}}{{2m}\choose{m}}$.  This is also the number of possible
blocks of size $2m+1$. 
We define the following three generating functions, which enumerate
central binomial coefficients of even order (E), of odd order (O),
and the Catalan numbers (C).  

\begin{eqnarray*}
E(x) & =  \frac{1}{\sqrt{1-4x}} & =  1 + 2x + 6x^2 + 20x^3 + 70x^4 + \dots \\
O(x) & =  \frac{\frac{1}{\sqrt{1-4x}}-1}{2x} & =  1 + 3x + 10x^2 + 35x^3 + 126x^4 + \dots \\
C(x) & =  \frac{1-\sqrt{1-4x}}{2x} & =  1 + x + 2x^2 + 5x^3 + 14x^4 + \dots
\end{eqnarray*}

The statement of the theorem is equivalent
to showing that $E(x)=\sum_{n\geq 0} A_{2n+1}x^{n}$ and $O(x)=\sum_{n\geq
0} A_{2n+2}x^{n}$, where we set $A_{n}:=\#\mathcal{A}_{n}$.  

Now a reachable permutation of even size $2n+2$ is the direct sum of
an indecomposable block of size $2i+1$ ($i\geq 0$) and a reachable
permutation of odd size $2(n-i)+1$. This translates into the
recursion/convolution
\[
A_{2n+2}=\sum_{i=0}^{n}c_{k}A_{2(n-i)+1}
\]

which is equivalent to $O(x)=E(x)C(x)$, and which is also easily verified from
the closed-form expressions for these generating functions.  Similarly,
a reachable permutation of odd-size $2n+1 $ is 
the direct sum of an indecomposable block of size $2i+1$ and a
reachable permutation of even size $2(k-i)$, corresponding to the
easily-verified equality of generating functions $E(x)=(1+xO(x))C(x)$.
This completes the proof.  
\end{proof}

Although the above proof seems natural enough from the structure of the
equivalence class $\mathcal{A}_{n}$, the simple form of the enumeration
as a single binomial coefficient begs the question of whether there is a
more direct (perhaps bijective) argument.  

The next theorem provides independent proofs of two results which
appeared 10 years ago in [CEHKN].  
\begin{theorem}\label{propo:P3bClassInv}
(a) $\Num^\twolines\left(n,\big\{ \{123,132,213\} \big\}\right) = \inv_{n}$, the
number of involutions of order $n$.

(b) $\#\Eq^\twolines\left(\pi,\big\{ \{123,132,213\} \big\}\right)$ is odd for all $n$ and for each $\pi \in S_n$.
\end{theorem}
\begin{proof}
Write each involution in $\tau \in \Inv_{n}\subseteq S_{n}$ canonically
as a product of 1-cycles and 2-cycles, with the elements increasing within each 2-cycle,
and with the cycles in decreasing order of largest element.  Omitting
the parentheses, we view the resulting word $D(\tau)$ as a
permutation.  Let $\mathcal{D}_{n}:=D(\Inv_{n})$ be the image of this map
(which is easily reversible by placing parentheses around the ascents of
$\sigma \in \mathcal{D}_{n}$).  We claim that this is a canonical set of
representatives for the equivalence classes of $S_{n}$ under $P_{3}= \{
\{123,132,213\} \}$ transformations.  

Each permutation $\pi\in S_{n}$ can be transformed to an element of
$\mathcal{D}_{n}$ as follows: 
if $n$ is at the front of $\pi$, it must stay there. (This corresponds
to having $n$ as a fixed point of the involution.) Otherwise, use $123
\rightarrow 132$ and $213 \rightarrow 132$ (at least one of which is
possible at each step) to push $n$ leftward into
position 2, which is as far as it will go.  The element which is thus
pushed into position 1 is the minimal element $m$ which was to the left of
$n$ to begin with. This is because $m$ can never trade
places with $n$ under the given operations, as 1 is left of 3 in all of
123, 132 and 213.  Leaving the leftmost 1-factor $n$ or 2-factor $mn$
fixed, proceed inductively among the remaining elements, at each step
moving the maximal remaining element as far left as possible.  The end
result of this deterministic procedure is a permutation $L(\pi)\in
\mathcal{D}_{n}$.  This shows that the number of $P_{3}$-equivalence
classes is at most $\inv_{n}=\#\mathcal{D}_{n}$.  

To show that they are the same, it remains to show that each $\pi$ can
be transformed to a \emph{unique} member of $\mathcal{D}_{n}$, or
equivalently that it is not possible to move from one member of
$\mathcal{D}_{n}$ to another using $P_{3}$-moves.  
We will prove this by induction on $n$. At the same time we will prove
statement (b) of the theorem. Assume as an induction hypothesis that
both statements have been demonstrated for $n-1$ and $n-2$. It is
straightforward to check the base cases by hand.  For $n=3$ the four
equivalence classes are $P_{3}$ and three singleton classes.  For $n=4$
the classes are $\{1234,1243,1324,2134,\mathbf{1423},1342,2143,3142,2314
\}$, $\{1342, 3124, \mathbf{1432}, 3142, 3214 \}$, $\{4123,
\mathbf{4132}, 4213 \}$, $\{2341, \mathbf{2431}, 3241\}$, and six
singletons: $\mathbf{\{ 2413\}, \{3412\},}$ $\mathbf{\{4312\},}$
$\mathbf{4231}$, $\mathbf{3421}$, $\mathbf{4321}$.  (The
elements in \textbf{bold} are the class representatives within
$\mathcal{D}_{n}$.)  

First note that if the largest element, $n$, is at the front of a
permutation, then it is immobile under $P_{3}$-moves. Thus the
equivalence classes split into two kinds: \emph{special} equivalence
classes, in which $n$ is at the front of each permutation in the class,
and \emph{ordinary} equivalence classes, in which $n$ is never at the
front.  Moreover it is obvious that the special equivalence classes for
$S_n$ correspond exactly to \emph{all} the equivalence classes for
$S_{n-1}$ upon deletion of the first elements; therefore, 
the truth of both (a) and (b) as they apply to the
\emph{special} equivalence classes follows by induction. 

Next we will look at the \emph{ordinary} equivalence classes. For
convenience of exposition, consider a (directed) graph in which the
vertices correspond to the permutations in $S_n$, and there is a blue
(directed) edge from $\pi$ to $\pi'$ if $\pi'$ can be obtained from
$\pi$ by applying $123 \rightarrow 132$, and similarly a red edge for
each $213 \rightarrow 132$, and a green edge for each $123 \rightarrow
213$. A blue edge just corresponds to a green edge followed by a red
one, and indeed the edges always appear in matched sets: the appearance
of a $213$ in a permutation implies an incident green edge pointing in
and a red edge pointing out, and also a blue edge making the chord of
this triangle (and similarly for appearances of $123$ and $132$). The
equivalence classes in which we are interested are the (undirected)
connected components of this graph.  
 
Now consider the forest of rooted trees which one obtains by taking only
those red and blue edges in which the element $n$ plays the role of the
``3''. The roots (i.e., sinks) of these trees are exactly the
permutations in which the $n$ has moved to position $2$, which is as
far left as it will go within an ordinary equivalence class. More generally,
if $\pi_{k}=n$, then $\pi$ lies on level
$k-2$ of the tree. (We can say that it has \emph{energy}
$E(\pi)=k-2=\pi^{-1}(n)-2$.) Note that blue and red edges reduce the
energy by one, while green edges leave it unchanged.  
Each vertex in this forest has either zero or two children, because if it
has a blue child (obtained by travelling backwards along a blue edge)
then it also has a red child, and vice versa.  

Each permutation $\pi$ lies on a unique directed path to the
root of its tree, which we will call the \emph{ground state} of $\pi$,
$g(\pi)$. Note that $g(\pi)_{2}=n$, while $g(\pi)_{1}$ is the smallest
element $m$ to the left of $n$ in $\pi$.
Because each node has either zero or two children, each rooted tree has
an odd number of nodes; indeed all of its level-sums are even except the
zeroth level sum, which corresponds to the root vertex (i.e., ground
state).

Now we will create larger classes as follows: declare two ground states
$\tau$ and $\sigma$ \emph{similar} if $\tau_{1}=\sigma_{1}$ and
$\tau_{3}\dotsb \tau_{n}$ is $P_{3}$-equivalent to $\sigma_{3}\dotsb
\sigma_{n}$ regarded (in the obvious way) as members of $S_{n-2}$.  For
$m\in [n-1]$ and $\nu \in \mathcal{D}_{n-2}$, let 
$\mathcal{K}(m,\nu)$ be the (disjoint!) union of all trees with similar ground
states $\tau $, where $\tau_{1}=m$ and $\tau_{3}\dotsb \tau_{n}$ is
$P_{3}$-equivalent to $\nu$.  Note that this gives us a total of
$(n-1)\inv_{n-2}$ equivalence classes, in agreement with the well-known
recursion: $\inv_{n}=\inv_{n-1}+(n-1)\inv_{n-2}$.  (The special
equivalence classes account for the first summand.)  

We claim that these larger classes $\mathcal{K}(m,\nu)$ are exactly (the
vertex sets of) the connected components of our directed graph; that is,
there are no directed edges in the graph which escape from one class to
another. Once this is shown, then by induction there is a unique member
of the class $\mathcal{D}_{n-2}$ of canonical permutations among the
ground states in a large component, to which we prepend $mn$ to obtain
the unique representative of $\mathcal{K}(m,\nu)$ within
$\mathcal{D}_{n}$.  

Furthermore, each $\mathcal{K}(m,\nu)$ will then be of odd size, because
each rooted tree has odd size, having all level-sums even except the one
corresponding to the ground states, and because the number of rooted
trees in the union is odd by the induction hypothesis for $n-2$.  

So suppose there is an edge (of any colour) from a $\pi \in
\mathcal{K}(m,\nu)$ to $\pi'\in \mathcal{K}(m',\nu')$.  
Since this move does not involve moving the largest element $n$,
$\pi$ and $\pi'$ have the same energy. Our goal is to show that
$m=m'$ and $\nu $ is $P_{3}$-equivalent to $\nu'$.  The former follows
from our earlier description of $m$ as the minimum element lying to the
left of $n$ in $\pi$, because $\pi$ and $\pi'$ have the same set of
elements to the left of $n$.  The latter requires an analysis of the
cases that can arise as $\pi$ and $\pi'$ move towards their ground
states in their respective trees.  

As the $n$ moves leftward through each of the two permutations
(following red and/or blue edges toward their respective ground states)
then it sometimes encounters identical elements and therefore has the
same effect; eventually it encounters the positions where the
difference lies,  having swept before it the
minimal intervening element, $b$. What happens from this point forward
depends on how $\pi$ and $\pi'$ differ, and the relative value of $b$.

To clarify the cases, let the three values where the difference was
applied be $d < f < h$, and designate $b$ by one of $C,E,G$ or $I$
(where $C<d<E<f<G<h<I$), depending on its relative order within the
factor.  For example, at some point along the path
from $\pi$ to $g(\pi)$ we may see a permutation containing the factor
$dfhCn$, while at the same energy level on the path from $\pi'$ to
$g(\pi')$ we see instead $dhfCn$ (having followed a blue edge), the two
permutations being otherwise identical. Advancing the element $n$ three further
steps to the left, we arrive in the first instance at $Cndfh$ and in the second
instance at $Cndhf$; the $n$ then continues forward all the way to
position $2$ (zero energy), making identical moves in each case. The
resulting ground states $g(\pi)$ and $g(\pi')$ differ only by a (blue)
move $dfh \rightarrow dhf$, so $\nu =\nu '$.  

Here is a table of the cases that arise given the four possible relative
values of $b$; blue edges are the composition of green with red.  

\begin{eqnarray*}
Input:  dfhCn  \hspace{5mm}{\color{green}\rightarrow} & fdhCn & {\color{red}\rightarrow}\hspace{5mm}  dhfCn \\
Output:  Cndfh  \hspace{5mm}{\color{green}\rightarrow} & Cnfdh & {\color{red}\rightarrow}\hspace{5mm}  Cndhf \\ 
\\
Input:  dfhEn  \hspace{5mm}{\color{green}\rightarrow} & fdhEn & {\color{red}\rightarrow}\hspace{5mm}  dhfEn \\
Output:  dnEfh  \hspace{5mm}{\color{green}\rightarrow} & dnfEh & {\color{red}\rightarrow}\hspace{5mm}  dnEhf \\
\\
Input:  dfhGn  \hspace{5mm}{\color{green}\rightarrow} & fdhGn & {\color{red}\rightarrow}\hspace{5mm}  dhfGn \\
Output:  dnfGh  \hspace{5mm}{=} & dnfGh & {\color{blue}\rightarrow}\hspace{5mm}  dnfhG \\
\\
Input:  dfhIn  \hspace{5mm}{\color{green}\rightarrow} & fdhIn& {\color{red}\rightarrow}\hspace{5mm}  dhfIn \\
Output:  dnfhI  \hspace{5mm}{=} & dnfhI & {=}\hspace{5mm}  dnfhI \\
\end{eqnarray*}

Examining this table shows that the classes $\mathcal{K}(m,\nu)$
containing $\pi$ and $\mathcal{K}(m',\nu')$ containing $\pi$ and $\pi'$
have $\nu $ $P_{3}$-equivalent to $\nu'$, which completes the proof.
\end{proof}

This result is particularly striking because the
equivalence relation has the same number of classes as Knuth
equivalence, yet the two relations are materially different.  For
example, for $n=3$, the equivalence classes for $P_{K}$ have sizes
1,1,2,2, whereas for $P_{3}=\big\{ \{123,132,213\} \big\}$ the sizes are
1,1,1,3. In fact the authors in \cite{Chinese} show that the corresponding
monoids (plactic and Chinese) share 
the same graded Hilbert series, and they obtain a partial recurrence for the
numbers $\#\Eq^{\twolines}(\iota_{n},P_{3})$.  

\begin{propo}[\cite{Chinese}, Cor. 4.3]\label{propo:chinese}  For $n$ odd, 
$\#\Eq^{\twolines}(\iota_{n},P_{3}) = n\cdot
\#\Eq^{\twolines}(\iota_{n-1},P_{3})$.  
\end{propo}

The recurrence for $n$ even appears still to be open.  

\section{Doubly adjacent transformations}\label{sec:dbl}

For completeness we include a brief treatment of the situation where
both indices and values are simultaneously constrained to be adjacent.
In this highly constrained situation, the permutations reachable from
the identity are easy to classify and enumerate in all cases.  Since all
the treatments are similar, we can wrap them up in one proposition.

As in the previous section, we have as yet no results related to the
enumeration of equivalence classes.

The statement of this proposition makes use of the \textit{Iverson bracket}\/: [S] is equal
to 1 if the statement S is true, and 0 otherwise.

\begin{propo}\label{propo:P*Id}
$\#\Eq^\oursquare(\iota_n,P_1)$ obeys the recurrence $a(n) = a(n-1) + a(n-2)$ 
with $a_1 = a_2 = 1$. (Fibonacci numbers $F(n)$, \cite[A000045]{OEIS}).

$\#\Eq^\oursquare(\iota_n,P_4)$ obeys the recurrence $a(n) = a(n-1) + a(n-3)$ 
with $a_0 = 0$, $a_1 = a_2 = 1$ (\cite[A000930]{OEIS}).

$\#\Eq^\oursquare(\iota_n,P_3) = F(n+1) - [n$ is even$]$.

$\#\Eq^\oursquare(\iota_n,P_5)$ obeys the recurrence $a(n) = a(n-1) + a(n-2) + a(n-3)$
with $a(0)=a(1)=a(2)=1$ (Tribonacci numbers, \cite[A000213]{OEIS}).

$\#\Eq^\oursquare(\iota_n,P_7) = T(n+2) - [n$ is even$]$,
where $T(n)$ obeys the recurrence 
$T(n) = T(n-1) + T(n-2) + T(n-3)$ with $T(0)=T(1)=0,
T(2)=1$. (Tribonacci numbers (with different initial conditions),
\cite[A000073]{OEIS}).  
\end{propo}
\begin{proof}
We begin by characterizing the various equivalence classes of $\iota_{n}$. In each
case, no element can move any further from its starting position then it
could via a single move.  The resulting classes are subsets of those
layered permutations which are direct sums of anti-identities of
dimensions either 1, 2 or 3, as follows: 

$P_1$ ($123 \leftrightarrow 132$): any direct sum of $\rho_1$ and
$\rho_2$ beginning with $\rho_{1}$; 

$P_4$ ($123 \leftrightarrow 321$): any direct sum of $\rho_1$ and
$\rho_3$; 

$P_3$ ($123 \leftrightarrow 132 \leftrightarrow 213$): any direct sum of
$\rho_1$ and $\rho_2$ including at least one $\rho_1$; 

$P_5$ ($123 \leftrightarrow 132 \leftrightarrow 321$): any direct sum of
$\rho_1$, $\rho_2$ and $\rho_3$ not beginning with $\rho_2$; 

$P_7$ ($123 \leftrightarrow 132 \leftrightarrow 213 \leftrightarrow
321$): any direct sum of $\rho_1$, $\rho_2$, $\rho_3$ with at least one
of odd dimension; 

In each case it is easy to see that the given class remains closed under
application of the appropriate operations. It is also easy in general to
see how to reach a given target permutation from $\iota_{n}$, especially
if we cast the block sizes in the language of regular expressions.  The
notation $\{xy\}$ means a single block of size either $x$ or $y$. An
asterisk following a number means zero or more copies of that number. An
asterisk following a string within $[\quad ]$ (not to be confused with
the Iverson brackets in the statement of the proposition) indicates zero
or more copies of that string.

$P_1$: The block sizes are $1\{12\}*=[12*]*$. Build each string of
blocks of the form $12*$ from right to left. 

$P_4$: The block sizes are $\{13\}*$; build each block freely. 

$P_3$: From any non-identity permutation with block sizes as described,
there is at least one instance of $21$ or $12$, which can be transformed
by one of the rules into a $111$.  The resulting permutation has one
fewer $\rho_{2}$.  Proceed inductively to transform all the $\rho_{2}$'s
(i.e., 2-blocks) to consecutive 1-blocks until the identity is reached.  

$P_5$: The block sizes are $[\{13\}2*]*$. First use $123 \rightarrow 132$ to build
    all the 2-blocks from right to left. Then use $123 \rightarrow 321$
    to place the 3-blocks.

$P_7$: Build the 2-blocks first, as in the case of $P_3$, and then place
the 3-blocks. 

One now verifies all the necessary base cases, as trivially $a_1=1$, $a_2=1$, 
and $a_3=$ the size of the non-singleton block of $P_{j}$.

As for the recurrences, for $n>3$:

$P_1$: $a_n = a_{n-1} + a_{n-2}$, by appending respectively a $\rho_1$ or a $\rho_2$.

$P_4$: $a_n = a_{n-1} + a_{n-3}$, by appending respectively a $\rho_1$ or a $\rho_3$.

$P_5$: $a_n = a_{n-1} + a_{n-2} + a_{n-3}$, by appending $\rho_1$, $\rho_2$ or $\rho_3$.

$P_3$: Count all direct sums of $\rho_1$ and $\rho_2$
(obviously Fibonacci) and then subtract 1 from the even terms to remove
the special case $2*$ (all blocks of size 2). 

$P_7$: Count all direct sums of $\rho_1$, $\rho_2$, $\rho_3$ to get
tribonacci numbers [A000073],
and subtract 1 from the even terms because block structure $2*$ is
disallowed. Alternatively, 
verify the recurrence $a_n = a_{n-2} + U_n$, where $U_n$ is the
$P_{5}$-recurrence [A000213], by
noting that a permutation in $\Eq^\oursquare(\iota_n,P_7)$ is either a $\rho_2$
prepended to a permutation in $\Eq^\oursquare(\iota_{n-2},P_7)$, or else
belongs to $\Eq^\oursquare(\iota_{n},P_5)$. 
\end{proof}

\section{Final Remarks \& Open Questions}\label{sec:open}
Our results in this paper are just a tractable subset of questions that
could be explored within these families of equivalence relations.  We
created the framework to easily allow for a number of extensions.  The
connections with familiar integer sequences, pattern-avoidance in
permutations, and important combinatorial bijections indicate the value
of further work. Possible directions for further study include:

\begin{enumerate}

\item Study the sizes (and characterise if possible) all equivalence
classes $\Eq^\ourstar(\pi,P)$, not just for the case $\pi
=\iota_{n}$. Corresponding to each equivalence relation is the multiset
of sizes of the equivalences classes, perhaps best considered as an
integer partition of $n!$.  Is the study of these of interest? 

\item Allow for more generality among the (set) partition $P$ of $S_{3}$
which defines our relations.  The authors in~\cite{PRW11} allow
substitution of patterns in $S_{3}$ where no element is fixed, but still
restrict to partitions $P$ consisting of exactly one non-singleton block
containing the identity $123$.  Although it seems unwieldy to work with
all $B(6)=203$ possible partitions of $S_{4}$, perhaps a different
restriction that forces greater symmetry among the relations would be
useful.  For example, the Knuth relations $P^{\twolines}_{K}=\big\{ \{
213,231\}, \{132,312\}\big\}$ are closed under reversal and
complementation.

\item Consider relations generated by partitions $P$ of $S_{k}$ for
$k>3$.  Here one definitely needs some conditions to restrict focus to
relations of particular interest, since the Bell number $B(4!)$ is
already far too large to handle all cases.  

\item Study in greater detail the structure of the graphs defined by
these relations.  What can one say about their degree sequences or
diameters?  How many moves are necessary in order to transform a given
$\pi $ to the identity?  

\end{enumerate}


\nocite{*}
\begin{thebibliography}{ABC99}
\providecommand{\natexlab}[1]{#1}
\providecommand{\url}[1]{\texttt{#1}}
\expandafter\ifx\csname urlstyle\endcsname\relax
  \providecommand{\doi}[1]{doi: #1}\else
  \providecommand{\doi}{doi: \begingroup \urlstyle{rm}\Url}\fi

\bibitem[AAK03]{AAKSimple} \textsc{M. H.~Albert, M. D.~Atkinson, and
M.~Klazar}, {\em The enumeration of simple permutations}, J. Integer
Sequences, {\bf 6}, Art. 03.4.4 (2003).  

\bibitem[SA08]{AssafSWD} \textsc{S.~Assaf}, \emph{A combinatorial
realization of Schur-Weyl duality via crystal graphs and dual
equivalence graphs},  FPSAC 2008, pp. 141--152,
Discrete Math. Theor. Comput. Sci. Proc., Nancy, France, 2008. 

\bibitem[SA07]{AssafDEG} \textsc{S.~Assaf}, \emph{Dual equivalence graphs, ribbon tableaux and Macdonald polynomials}, Ph.D. Thesis, UC Berkeley, 2007

\bibitem[Bon04]{BonaPerm}
{\sc M.~Bona}, {\em Combinatorics of Permutations}, Chapman \&  Hall/CRC, 2004.

\bibitem[Bon06]{BonaWalk}
{\sc M.~Bona}, {\em A Walk Through Combinatorics, 2nd Ed.}, 
  World Scientific Publishing Co., 2006.

\bibitem[CEHKN]{Chinese}
{\sc J.~Cassaigne, M.~Espie, D.~Krob, J.-C.~Novelli, F.~Hivert}, {\em
The Chinese Monoid}, Int'l. J. Algebra and Comp., \textbf{11} \#3 (2001),
301--334.  

\bibitem[CK08]{ClaKit}
{\sc A.~Claesson and S.~Kitaev}, {\em Classification of bijections
between 321- and 132-avoiding permutations}, S\'eminaire Lotharingien de
Combinatoire \textbf{60} (2008), Article B60d (electronic).  

\bibitem[Com72]{Comn!} L.~Comtet, Sur les coefficients de l'inverse de
la series formelle $\sum n! t^{n}$ Comptes Rend. Acad. Sci. Paris,
\textbf{A 275} (1972), 569-572.

\bibitem[Com74]{ComAdv} \textsc{L.~Comtet} \emph{Advanced
Combinatorics}, Reidel, 1974, pp. 84 (\#25), 262 (\#14) 
and 295 (\#16).

\bibitem[Hai92]{HaimanDEq}
{\sc M.~Haiman}, {\em Dual equivalence with applications, including a
conjecture of {P}roctor}, Discrete Math., {\bf 99} (1992), pp.~79--113. 

\bibitem[Knu70]{Knuth}
{\sc D.~Knuth}, {\em Permutations, matrices and generalized Young
tableaux}, Pacific J. Math., {\bf 34} (1970), pp.~709--727.

\bibitem[LLT02]{LLTPlactic} \textsc{A.~Lascoux, B.~Leclerc, and
J.-Y.~Thibon}, \emph{The Plactic Monoid}, Ch.~5 in M.~Lothaire
\textit{Combinatorics on Words}, Encyclopedia of Math and its App.,
Cambridge U. Press, 2002, pp. 164--196.  (Also downloadable from 
\texttt{http://www-igm.univ-mlv.fr/\~{}jyt/articles.html}.)  

\bibitem[LS81]{LSPlactic} \textsc{A.~Lascoux and
M.~P.~Sch\"utzenberger}, ``Le mono\"ide plaxique'', in Non commutative
structures in Algebra and Combinatorics, Quaderni della Ricerca
Scientifica del CNR, Roma, (1981). 

\bibitem[LPRW10]{LPRW} S.~Linton, J.~Propp, T.~Roby, and
J.~West. Equivalence Relations of Permutations Generated by Constrained
Transpositions . DMTCS Proceedings, North America, July 2010. 
(Retrieved: 26 Aug. 2011.) \hfil \break
http://www.dmtcs.org/dmtcs-ojs/index.php/proceedings/article/view/dmAN0168

\bibitem[OEIS]{OEIS} N.J.A.~Sloane, \emph{The On-line Encyclopedia of Integer Sequences}, (Retrieved: 26 Aug. 2011.)
\texttt{http://www.research.att.com/\urltilde njas/sequences/}.  

\bibitem[PRW11]{PRW11} \textsc{A.~Pierrot, D.~Rossin, and J.~West}, \emph{Adjacent
transformations in permutations},  FPSAC 2011 Proceedings, Discrete
Math. Theor. Comput. Sci. Proc., 2011.  
(Retrieved 23 August 2011.)\hfil \break
\texttt{http://www.dmtcs.org/dmtcs-ojs/index.php/proceedings/article/view/dmAO0167/3638}.


\bibitem[Pri97]{PriceLayered} \textsc{A. L.~Price}, \emph{Packing
Densities of Layered Patterns}, Ph.D. Thesis, University of
Pennsylvania, 1997. 

\bibitem[EC1]{StanleyEC1}
{\sc R.~Stanley}, {\em Enumerative Combinatorics Volume 1}, no.~49 in Cambridge
  Studies in Advanced Mathematics, Cambridge University Press, 1999.

\bibitem[EC2]{StanleyEC2}
{\sc R.~Stanley}, {\em Enumerative Combinatorics Volume 2}, no.~62 in Cambridge
  Studies in Advanced Mathematics, Cambridge University Press, 1999.
\newblock With appendix 1 by Sergey Fomin.

\bibitem[CatAdd]{StanleyCatAdd}
{\sc R.~Stanley}, {\em Catalan Addendum}, (version of 22 October 2011;
89 pages). (Retrieved 24 October 2011.)
http://www-math.mit.edu/~rstan/ec/catadd.pdf

\bibitem[Stro93]{StromLayered} \textsc{W.~Stromquist}, \emph{Packing
Layered Posets Into Posets}, preprint, 1993.  (Retrieved: 26 Aug. 2011.)
\texttt{http://walterstromquist.com/}.  

\end{thebibliography}
\end{document}